\documentclass[12pt]{amsart}

\usepackage{amsmath, amssymb, amsthm, amsfonts, amscd}

\input xy
\xyoption{all}

\usepackage{hyperref}
\usepackage{graphicx}
\usepackage{color}
\usepackage{tikz-cd}
\usepackage{rotating}

\numberwithin{equation}{section}

\newtheorem{theorem}[equation]{Theorem}
\newtheorem*{theorem*}{Theorem}
\newtheorem{corollary}[equation]{Corollary}
\newtheorem{lemma}[equation]{Lemma}

\newtheorem{proposition}[equation]{Proposition}

\theoremstyle{definition}
\newtheorem{definition}[equation]{Definition}
\newtheorem{remark}[equation]{Remark}
\newtheorem{example}[equation]{Example}
\newtheorem*{example*}{Example}

\newtheorem{notation}[equation]{Notation}

\usepackage[margin=1in,marginparwidth=0.8in, marginparsep=0.1in]{geometry}

\def\CC{\mathbb{C}}

\def\LL{\mathbb{L}}

\def\PP{\mathbb{P}}

\def\RR{\mathbb{R}}

\def\TT{\mathbb{T}}

\def\ZZ{\mathbb{Z}}

\def\cA{\mathcal{A}}

\def\cH{\mathcal{H}}

\def\cL{\mathcal{L}}

\newcommand\tilX{\widetilde{X}}

\newcommand{\Coh}{\textup{Coh}}

\newcommand\Fuk{\textup{Fuk}}

\newcommand{\Ind}{\textup{Ind}}

\newcommand{\Perf}{\textup{Perf}}

\newcommand{\QCoh}{\textup{QCoh}}

\newcommand\res{\textup{res}}

\newcommand\Spec{\textup{Spec}}

\newcommand\Hom{\textup{Hom}}

\newcommand\nc{\newcommand}
\nc\mush{\mu\on{sh}}
\newcommand*{\isoarrow}[1]{\arrow[#1,"\rotatebox{90}{\(\sim\)}"]}

\nc\oX{\overline{X}}
\nc\oD{\overline{D}}
\nc\oT{\overline{T}}
\nc\bsig{\overline{\sigma}}
\nc\Pprop{\on{Perf}_{\on{prop}}}
\nc\Tors{\on{Tors}}

\nc\on{\operatorname}
\nc\ol{\overline}
\nc\ul{\underline}
\nc\Bl{\on{Bl}}
\nc\Conv{\on{Conv}}
\nc\Int{\on{Int}}
\nc\Pc{\mathring{P}}

\nc\oo{\infty}

\nc\Cone{\on{Cone}}
\nc\ssupp{\mathit{ss}}
\nc\risom{\stackrel{\sim}{\to}}
\nc\Sh{\on{Sh}}
\nc\un{\diamondsuit}
\nc\orient{\mathit{or}}
\nc\sing{\mathit{sing}}
\nc\MF{\on{MF}}
\nc\Log{\on{Log}}
\nc\Arg{\on{Arg}}
\nc\inthom{\mathit{Hom}}
\nc\colim{\varinjlim}
\nc\Dmod{\calD}
\nc\hh{\heartsuit}
\nc\mmod{\on{-mod}}
\nc\wsh{\Sh^w}
\nc\wmsh{\mu\Sh^w}

\nc\Conf{\on{Conf}}
\nc\pt{\on{pt}}
\nc\ovi{\overline{i}}

\nc\conv{\mathit{conv}}
\nc\trad{\mathit{inf}}
\nc\wrap{\mathit{wr}}
\nc\tP{\widetilde{P}}
\nc\tX{\widetilde{X}}
\nc\tphi{\widetilde{\varphi}}
\nc\Lmsing{\mathbb{L}^m_{\on{deg}}}
\nc\defpot{\widetilde{\varphi}}
\nc\tcA{\widetilde{\mathcal{A}}}
\nc\rT{\mathrm{T}}
\nc\rS{\mathrm{S}}

\begin{document}

\title{Local mirror symmetry via SYZ}
\author{Benjamin Gammage}
\subjclass{53D37, 14J33}
\address{Department of Mathematics, Harvard University,
1 Oxford St., Cambridge MA 02138}
\email{gammage@math.harvard.edu}

\maketitle

\begin{abstract}
In this note, we explain how mirror symmetry for basic local models in the Gross-Siebert program can be understood through the non-toric blowup construction described by Gross-Hacking-Keel. This is part of a program to understand the symplectic geometry of affine cluster varieties through their SYZ fibrations.
\end{abstract}
\setcounter{tocdepth}{1}

\section{Introduction}
In the Gross-Siebert program, a pair of dual spaces is described by the combinatorial data of an integral affine manifold with singularities, which is supposed to be understood as the base of a Lagrangian torus fibration with singularities. Locally near a smooth point of the base, 
the total space looks like the cotangent bundle of a torus, projecting to its cotangent fibers.

Near singular points in the base, the behavior of a torus fibration can become more complicated as torus fibers acquire singularities. However, in each dimension $N$, there is a finite list of such singularities $\mathfrak{s}_{n,m}$ (with $n+m=N$) which are known to serve as building blocks for a wide class of integral affine manifolds. Above the singularity $\mathfrak{s}_{n,m},$ the symplectic manifold is described by the local model
\[
  X_{n,m}:=\{z_0\cdots z_n = 1 + u_1 + \cdots u_m\mid z_i\in \CC,u_j\in\CC^\times\},
\]
equipped with singular Lagrangian torus fibration $X_{n,m}\to \RR^{n+m}$ by
\begin{equation}\label{eq:SYZ-intro}
  (z_0,\ldots,z_n,u_1,\ldots,u_m)\mapsto (|z_0|^2-|z_1|^2,\ldots,|z_0|^2-|z_n|^2,\log|u_1|,\ldots,\log|u_m|).
\end{equation}
In the case $n=1,$ this SYZ fibration on the conic bundle $X_{1,m}$ was studied extensively in \cite{AAK}.

In this paper, we will show (in \S \ref{sec:trop}) that the spaces $X_{n,m}$ may be described using a generalization of the non-toric blowup construction of \cite{GHK1,GHK2}. 
The mirror to that construction was described in detail in \cite{GL} (following earlier works \cite{STW,Nad-mut,PT}); by generalizing that mirror construction to the case of $X_{n,m},$
we can deduce a mirror symmetry equivalence between these spaces:
\begin{theorem}\label{thm:main}
 There is an equivalence of categories $\Fuk(X_{m,n})\cong \Coh(X_{n,m}).$
\end{theorem}

\begin{remark}Theorem \ref{thm:main} as formulated has long been expected; a proof of this theorem from a more traditional Fukaya-categorical perspective, providing a more detailed correspondence between homological mirror symmetry and SYZ geometry, will appear in \cite{AS2x}. Many special cases of Theorem \ref{thm:main} appear in the literature, including the case $m=-1$ (the $(n-1)$-dimensional pants, studied in \cite{Lekili-Polishchuk} and from the perspective of Lagrangian skeleta in \cite{Nwms}), and the case $n=1$ (the substrate of the basic ``generalized cluster transformation,'' studied in \cite{Nad-mut,PT}) -- indeed, in some sense the content of this paper consists in combining those two calculations. (Other special cases of this theorem can be found in \cite{CPU,Pom21}.)
 
 But the broader novelty in this work is in relating the Lagrangian skeleton of $\Fuk(X_{m,n})$ to the geometry of its SYZ fibration \eqref{eq:SYZ-intro}. This is part of a larger program, to be discussed in \S \ref{sec:trop}, to understand Lagrangian skeleta and mirror symmetry for affine cluster varieties via their SYZ fibrations. 
\end{remark}

On the A-side, the main input to this theorem is a calculation of the Lagrangian skeleton $\LL_{m,n}$ of the Weinstein manifold $X_{m,n}.$
Let $\Sigma_{\CC^n\times (\CC^\times)^m}$ be the fan of cones in $\RR^{n+m}$ spanned by subsets of the first $n$ basis vectors $e_1,\ldots,e_n,$ identifying each cone $\sigma$ with the corresponding subset of $[n].$
Introduce the conic Lagrangian 
\begin{equation}\label{eq:FLTZ}
  \LL_{\CC^n\times (\CC^\times)^m} := \bigcup_{\sigma\subset [n]} \sigma^\perp\times \sigma\subset \RR^{n+m}/\left(\ZZ+\frac12\right)^{n+m} \times \RR^{n+m} = \rT^* T^{n+m},
\end{equation}
where in the first factor we write $\sigma^\perp$ for the image of the subspace orthogonal to $\sigma$ under the projection $\RR^{n+m}\to T^{n+m}.$
This Lagrangian, studied in \cite{FLTZ1,FLTZ2,FLTZ3}, following earlier work \cite{Bo-ccc}, is known \cite{GS17,Zhou-skel} to be the skeleton of the Liouville-sectorial mirror to the toric variety $\CC^n\times (\CC^\times)^m$ \cite{Ku,Zhou-ccc,V}, and its boundary
\[
  \partial \LL_{\CC^n\times (\CC^\times)^m} = \bigcup_{\emptyset\neq \sigma\subset [n]} \sigma^\perp\times \partial^\infty\sigma\subset \rS^* T^{n+m}
\]
is mirror to the toric boundary $\partial \CC^n\times (\CC^\times)^m$ \cite{Nwms, GS17}.

\begin{remark}
  In \cite{GS17}, we avoided the translation $\ZZ\mapsto \ZZ+\frac12$ appearing in the definition of the quotient torus $T^{n+m}$ in \eqref{eq:FLTZ} by changing the sign of the constant coefficient in $1+u_1+\cdots + u_m.$ In this paper we prefer the more symmetrical convention for this polynomial.
\end{remark}

Observe that $\LL_{\CC^n\times (\CC^\times)^m}$ is globally a product 
\[
  \LL_{\CC^n\times (\CC^\times)^m} = \LL_{\CC^n}\times T^m\subset \rT^*T^n\times \rT^*T^m.
\]
We will be interested in a second Lagrangian which involves the degeneration of the $T^m$ factor in that product. Let $\Lmsing\subset \CC^{m+1}$ be the ``Lagrangian $m$-torus degeneration,''
\[
  \Lmsing := \{(z_0,\ldots,z_m)\mid |z_0| = \cdots = |z_m|, \prod_{i=0}^m z_i\in \RR_{\geq 0} \},
\]
studied in \cite{Nad-CnW,G20} as the skeleton of a mirror to a general linear hypersurface in $(\CC^\times)^m,$ and in \cite{Nad-mut,PT} as part of higher-dimensional cluster theory. Its Legendrian boundary $\partial \Lmsing$ is the $m$-torus $T^m.$

Our main result on the symplectic geometry of symplectic manifold $X_{m,n}$ is a calculation of its Lagrangian skeleton $\LL_{m,n}.$

\begin{theorem}\label{thm:main2}
  The Lagrangian skeleton $\LL_{m,n}$ of $X_{m,n}$ has a cover
  \[
    \LL_{m,n}=\cL_1\cup_{\cL}\cL_2,
  \]
  where $\cL_1\cong \LL_{\CC^n\times(\CC^\times)^m},$ $\cL_2\cong (\partial \LL_{\CC^n})\times \Lmsing,$ and their intersection is $\cL\cong(\partial \LL_{\CC^n})\times \RR \times T^m.$
\end{theorem}
The proof of Theorem \ref{thm:main2}, essentially a combination of the calculations in \cite{Nwms,PT}, will be given in \S\ref{sec:skel-calc}. First, In \S\ref{sec:cat-gluing}, we give a proof of Theorem \ref{thm:main} assuming Theorem \ref{thm:main2}.
Further discussion on the meaning of these results, and their relations to cluster theory and a generalization of \cite{GL}, will be postponed to \S\ref{sec:trop}.

We conclude the introduction with an example recalling how our construction recovers the geometry of the ``cluster local model'' $X_{1,1}.$
\begin{example}
  Let $m=n=1,$ so that $\mathfrak{s}_{1,1}$ is the focus-focus singularity, whose local model
  \[
  X_{1,1} = \{z_0z_1 = 1 + u\} = \CC^2 \setminus \{z_0z_1 = 1\}
  \]
  was first studied from the perspective of mirror symmetry in \cite[\S 5]{Aur07}. The Weinstein manifold $X_{1,1}$ admits a Lagrangian skeleton $\LL$ which is the union of a torus $T=T^2$ and a disk whose boundary is glued along a primitive homology class of $T$. In line with Theorem \ref{thm:main2}, we can think of this as a union of two open pieces: a torus with a cylinder attached (the mirror to $\CC\times \CC^\times$) and a disk collapsing the boundary circle of that cylinder to a point.
\end{example}

\section{Gluing microsheaf categories}\label{sec:cat-gluing}
In this section, we show how Theorem \ref{thm:main} follows from Theorem \ref{thm:main2} by gluing together prior microsheaf calculations. We collect those calculations first.

\begin{notation}
    Throughout this section, we write 
    \[
    R:=\CC[z_1,\ldots,z_n,u_1^\pm,\ldots,u_m^\pm].
    \]
    We will also write $\Pi:=z_1\cdots z_n$ and $\Sigma:=1+u_1+\cdots + u_m$ for the respective product and sum, and we will be interested in the quotient rings $R/\Pi,R/\Sigma,$ and $R/(\Pi,\Sigma).$
\end{notation}

\begin{notation}Following the convention of \cite{GS20}, we will write $\mu\Sh$ for the sheaf of microlocal sheaf categories defined in \cite[Ch. 6]{Kashiwara-Shapira} (or technically speaking, the sheafification of the presheaf of categories defined there) and $\mush$ for the sheaf of categories defined in \cite{NS20}. The former of these is defined within a cotangent bundle or cosphere bundle; the latter is defined on more general symplectic or contact manifolds equipped with polarization data. (It is the latter sheaf which is related to the Fukaya category by \cite{GPS3}, but some of the calculations we cite, predating \cite{NS20}, are expressed in the older language.)
We take these to be valued in the presentable category $\CC\mmod$ of $\CC$-modules, although more general choices of coefficients are possible.
\end{notation}

\begin{theorem}
    There is a commutative diagram
\begin{equation}\label{eq:sheaves}
\begin{tikzcd}
    \Sh_{\LL_{\CC^n}}(T^n)\ar[r]\isoarrow{d} & \mu\Sh(\partial \LL_{\CC^n})\isoarrow{d}\\
    \QCoh(\CC^n)\ar[r] & \QCoh(\partial \CC^n),
\end{tikzcd}
\end{equation}
where the vertical maps are equivalences, and the bottom horizontal map is the pullback to the toric boundary of $\CC^n$ on the first factor (and the identity on the second factor).
\label{thm:basic-calculation}
\end{theorem}
\begin{proof}
    The left-hand isomorphism
is a special case of \cite[Theorem 1.3]{Ku} but the idea essentially goes back to \cite{Bo-ccc}. (See also \cite{FLTZ1}.) This is an instance of the ``Basic calculation'' described at \cite[\S 7.1]{GS17}: see there for more references. The existence of the right-hand isomorphism making the square commute is \cite[Theorem 7.13]{GS17}.
\end{proof}
Theorem \ref{thm:basic-calculation} is stated in the language of sheaves and microlocalization within the cotangent bundle $\rT^*T^n.$ For our purposes, it will be necessary to rephrase this in terms of the theory of microlocal sheaves in a general Weinstein manifold.
\begin{lemma}\label{lem:KS-equals-NS}
There is a commutative diagram
\begin{equation}\label{eq:KS-equals-NS}
\begin{tikzcd}
    \Sh_{\LL_{\CC^n}}(T^n)\ar[r]\isoarrow{d} & \mu\Sh(\partial \LL_{\CC^n})\isoarrow{d}\\
    \mush(\LL_{\CC^n})\ar[r]&\mush(\partial \LL_{\CC^n})
\end{tikzcd}
\end{equation}
where the vertical maps are equivalences, and the bottom row is computed using the cotangent fiber polarization of $\rT^*T^n.$
\end{lemma}
\begin{proof}
    For a Legendrian $L\subset S^*X$ in a cosphere bundle, there is a canonical equivalence (\cite[Remark 9.5]{NS20}; see also \cite[Corollary 4.13]{CKNS}) between $\mu\Sh(L)$ and $\mush(L),$ where the latter is computed in the cotangent fiber polarization. However, $\LL_{\CC^n}$ starts life as a conic Lagrangian in a cotangent bundle, rather than a Legendrian in a cotangent bundle. 
    
    As explained in \cite[\S 4E]{GS20}, for $L\subset \rT^*X$ a conic Lagrangian in a cotangent bundle, there is not a canonical equivalence between $\mu\Sh(L)$ and $\mush(L),$ where again the latter is computed in the cotangent fiber polarization. However, \cite[\S 4E]{GS20} constructs a (noncanonical) equivalence between these, which supplies the vertical equivalences in \eqref{eq:KS-equals-NS}.
\end{proof}
\begin{theorem}[{\cite[Corollary 1.8]{Nad-CnW}}]\label{thm:deg-calc}
	There is a commutative diagram
 \begin{equation}\label{eq:lmdeg-diagram}
 \begin{tikzcd}
     \mush(\Lmsing)\ar[r]\isoarrow{d}&\mush(\partial \Lmsing)\isoarrow{d}\\
     \CC[u_1^\pm,\ldots,u_m^\pm]/\Sigma\mmod\ar[r] & 
     \CC[u_1^\pm,\ldots,u_m^\pm]\mmod,
 \end{tikzcd}
 \end{equation}
 where the bottom horizontal map is restriction of scalars, and
 the polarization data for computing $\mush$ comes from the Legendrian lift of $\Lmsing$ to $S^*(\RR^{m+1}\times \RR)$ discussed in \cite[\S 3.1]{Nad-CnW}.
\end{theorem}
\begin{remark}
    As \cite{Nad-CnW} predates the theory of microlocal sheaves in a general Weinstein manifold constructed in \cite{NS20}, the above calculation is not stated in the language of polarizations; rather, the category of microlocal sheaves on $\Lmsing$ (which is not a conic Lagrangian in a cotangent bundle) is defined to be the category of microlocal sheaves on its Legendrian lift from $\rT^*(\RR^{m+1})$ to $J^1(\RR^{m+1})\subset S^*(\RR^{m+1}\times \RR).$
    Note that as the identification $\CC^{m+1}\simeq \rT^*(\RR^{m+1})$ does not respect the natural Liouville forms, a priori $\Lmsing$ lifts to a Legendrian in the ``wrong'' contact structure on $J^1(\RR^{m+1}),$ viz., the one coming from its identification with the contactization $\CC^{m+1}\times \RR$. However, as explained in \cite[Remark 3.3]{Nad-CnW}, since the difference of the two Liouville forms is exact, the two contact structures may be related by a contactomorphism. 
\end{remark}

We now equip the Lagrangian $\LL_{m,n}=\cL_1\cup_\cL \cL_2$ with the following polarization: 
\begin{itemize}
    \item  On $\cL_1 = \LL_{\CC^n}\times (\CC^\times)^m = \LL_{\CC^n}\times (\partial \Lmsing),$ we take the product of the cotangent fiber polarization on $\LL_{\CC^n}\subset \rT^*T^n$ and the polarization on $\partial \Lmsing$ described in \cite[\S 3.1]{Nad-CnW}.
    \item On $\cL_2 = (\partial \LL_{\CC^n}) \times \Lmsing,$ we again take the product of the cotangent fiber polarization and the polarization described in \cite{Nad-CnW}.
\end{itemize}
As the restrictions of these two polarizations to $\cL$ are canonically identified, the above does indeed define a polarization on the total Lagrangian $\LL_{m,n}.$

\begin{proposition}
    In the above polarization, the category $\mush(\LL_{m,n})$ of microlocal sheaves on $\LL_{m,n}$ may be computed as a homotopy limit
    \begin{equation}\label{eq:prop-limit-diagram1}
    \mush(\LL_{m,n})\simeq \varprojlim
    \left(
        R\mmod \to R/\Pi\mmod \gets R/(\Pi,\Sigma)\mmod
    \right)
    \end{equation}
\end{proposition}
\begin{proof}
    By descent, the category $\mush(\LL_{m,n})$ may be computed as the homotopy limit of the diagram
    \[
        \mush(\cL_1)\to \mush(\cL) \gets \mush(\cL_2),
    \]
    which we may rewrite as
    \[
    \mush(\LL_{\CC^n}\times (\partial \Lmsing)) \to
    \mush((\partial \LL_{\CC^n}) \times (\partial \Lmsing)) \gets
    \mush((\partial \LL_{\CC^n})\times \Lmsing),
    \]
    with the polarizations as given above. Thus, by a K\"unneth theorem, the middle term splits as a tensor product, with the maps being given by the corresponding map from \eqref{eq:KS-equals-NS} or \eqref{eq:lmdeg-diagram} on one factor and the identity on the other factor. From Theorem \ref{thm:basic-calculation}, Lemma \ref{lem:KS-equals-NS}, and Theorem \ref{thm:deg-calc}, we conclude that the above diagram is equivalent to the diagram
    \[
      \QCoh(\CC^n)\otimes \QCoh((\CC^\times)^{m}) \to \QCoh(\partial \CC^n)\otimes \QCoh((\CC^\times)^m) \gets \QCoh(\partial \CC^n)\otimes \QCoh\left(\{\Sigma=0\}\right),
    \]
    which agrees with the limit diagram on the right-hand side of \eqref{eq:prop-limit-diagram1}.
\end{proof}

Now we compute the above homotopy limit algebraically.
\begin{proposition}\label{prop:hlim-computation}
    Let $R'$ denote the ring $R':=R[z_0]/(z_0\Pi=\Sigma).$ Then
    there is an equivalence of categories
    \begin{equation}\label{eq:pullback-modcats}
    R'\mmod\xrightarrow{\sim}
    \varprojlim\left(R\mmod \to R/\Pi\mmod\gets R/(\Pi,\Sigma)\mmod\right)
    \end{equation}
    describing the category of $R'$-modules as the indicated homotopy fiber product.
\end{proposition}
\begin{proof}
    An object of the right-hand side of \eqref{eq:pullback-modcats} can be presented as a triple $(M,N,\varphi)$ where $M$ is an $R$-module, $N$ is an $R/(\Pi,\Sigma)$-module, and $\varphi$ is an equivalence of $R/\Pi$-modules between $M\otimes_R R/\Pi$ and $\res^{R/\Pi}_R(N)$, where we write $\res$ for restriction of scalars.
    From this perspective, the functor in \eqref{eq:pullback-modcats} sends an $R'$-module $M'$ to the triple 
    \[
    (\res_{R'}^{R}M', \res_{R'}^{R}(M')\otimes_{R}R/(\Pi,\Sigma),\varphi),
    \]
    where $\varphi$ is the evident isomorphism (coming from the fact that $\{\Pi=0\}$ and $\{\Pi=\Sigma=0\}$ describe the same closed subscheme of $\Spec(R')$).
    
    To see that this functor is an equivalence,
    we will want
    a different description of the objects of the limit category. We observe that the data of a triple $(M,N,\varphi)$ as above is equivalent to the data of $(M,\xi_0)$ where $M$ is as before and $\xi_0$ is a nulhomotopy of the action of $\Sigma$ on $M\otimes_{R}R/\Pi.$ (Given $(M,N,\varphi),$ we may use $\varphi$ to identify $N$ as an $R/\Pi$-module with $M\otimes_RR/\Pi,$ and $\xi_0$ is the data of a lift of the $R/\Pi$-module structure to an $R/(\Pi,\Sigma)$-module.)
    
    Taking a semi-free resolution of $R/\Pi$ as $R[\eta]$ with $|\eta|=-1$ and $d\eta = \Pi,$ we may write $M\otimes_R R/\Pi\simeq M\otimes_R R[\eta]$ as the complex $M[\eta] = (M\xrightarrow{\Pi}M),$ which is a semi-free $R$-module with basis $\eta,1$. If $M$ is a discrete (i.e., concentrated in degree 0) $R$-module, then on the (vertical) map of complexes (with differentials written horizontally)
    \[
    \begin{tikzcd}
        M\ar[r,"\Pi"]\ar[d,"\Sigma"'] & M\ar[d,"\Sigma"]\\
        M \ar[r,"\Sigma"] & M,
    \end{tikzcd} \quad\text{ the nulhomotopy $\xi_0$ is given by a diagonal map }\quad
     \begin{tikzcd}
        M\ar[r,"\Pi"]\ar[d,"\Sigma"'] & M\ar[d,"\Sigma"]\ar[dl,dashed,"z_0"']\\
        M \ar[r,"\Pi"] & M
    \end{tikzcd}
    \]
    satisfying $z_0 \Pi= \Pi z_0 = \Sigma.$

    In general (for $M$ not necessarily a discrete module), 
    the nulhomotopy $\xi_0,$ as a degree-$(-1)$ endomorphism of $\Cone(M\xrightarrow{\Pi}M),$ has two components $\alpha$ and $z_0$, illustrated as the dashed maps in the diagram
    \[
    \begin{tikzcd}
        M\ar[r,"\Pi"] \ar[d,dashed,"\alpha"']& M\ar[d,dashed,"\alpha"]\ar[dl,dashed,"z_0"']\\
        M\ar[r,"\Pi"]&M,
    \end{tikzcd}
    \]
    where $\alpha$ is a degree-$(-1)$ cochain (not necessarily a cocycle) in the mapping complex $\Hom_{R}(M).$ In terms of $\alpha$ and $z_0,$ the condition that $\xi_0$ is a nulhomotopy of $f$ is the relation
    \[
    z_0\Pi + d\alpha = f,
    \]
    where $d$ is the internal differential on $M$.
    
    In other words, the data of $\xi_0$ is equivalent to the data of a lift of the $R$-module structure on $M$ to an $R[z_0,\alpha]$-module structure, where $\alpha$ has degree-$(-1)$ and satisfies $d\alpha = f - z_0\Pi.$ But the ring $R[z_0,\alpha]$ is equivalent (as a cofibrant replacement) to $R',$ so we conclude that the category of pairs $(M,\xi_0)$ is equivalent to the category of $R'$-modules. It is straightforward to see that the functor \eqref{eq:pullback-modcats} realizes the canonical equivalence.
\end{proof}

Assuming Theorem \ref{thm:main2}, the proof of Theorem \ref{thm:main} follows immediately from the above:
\begin{proof}[Proof of Theorem \ref{thm:main}]
    Combining \eqref{eq:prop-limit-diagram1} and \eqref{eq:pullback-modcats}, we have an equivalence $\mush(\LL_{m,n})\simeq \QCoh(X_{n,m}).$ Theorem \ref{thm:main2} asserts that $\LL_{m,n}$ is the skeleton of the Weinstein manifold $X_{m,n},$ and we may therefore apply \cite[Theorem 1.4]{GPS3} to produce an equivalence $\mush(\LL_{m,n})\simeq \Ind\Fuk(X_{m,n})$ (where each is defined with the stable polarization constructed above).
    We deduce Theorem \ref{thm:main} by combining these equivalences and passing to compact objects.
\end{proof}

\section{Calculating the skeleton}\label{sec:skel-calc}
Let $\varphi:X_{m,n}\to \RR$ be a Morse-Bott K\"ahler potential for the Stein manifold $X_{m,n},$ so that its Stein K\"ahler form is $\omega=dd^c \varphi,$ which has primitive $\lambda = d^c\varphi.$ 
Morse theory gives a handle presentation of the manifold $X_{m,n},$ where only critical points of index up to $m+n$ occur, and the union of stable manifolds for the critical points is a singular Lagrangian subset $\LL\subset X_{m,n},$ the {\em skeleton} of $X_{m,n}.$ These Lagrangian skeleta are of great interest thanks to the fact that the Fukaya category $\Fuk(X)$ of a Stein manifold $X$ may be computed locally over its Lagrangian skeleton $\LL$ \cite{GPS1,GPS2,GPS3,NS20}.

Let $\varphi_1:\CC^{m+1}\to \RR$ and
and $\varphi_2:(\CC^\times)^n\to \RR$ denote the respective functions
\[
\varphi_1(z):=\sum_{i=1}^m\left(|z_0|^2-|z_i|^2\right)^2, \qquad
\varphi_2(u):= \sum_{j=1}^n(\log|u_i|+\ell)^2,
\]
where we have fixed some $\ell\ll 0.$
We will write
\[
\varphi(z,u):=\varphi_1(z)+\varphi_2(u)
\]
for the sum of these functions.

The function $\varphi$ is natural from the perspective of the SYZ fibration $\pi:X_{m,n}\to \RR^{m+n}$ defined by
\begin{equation}\label{eq:SYZ}
  \pi(z,u) := (|z_0|^2-|z_1|^2,\ldots,|z_0|^2-|z_m|^2, \log|u_1|,\ldots,\log|u_m|),
\end{equation}
so that $\varphi$ can be written as $\varphi(z,u)=d(\pi(z,u),(\vec{0},\vec{\ell})),$ where $d$ is the distance function on $\RR^{n+m}.$ In what follows, we will write $(\eta_1,\ldots,\eta_m,\xi_1,\ldots,\xi_m)$ for the natural coordinates on this SYZ base $\RR^{n+m}.$
Note that the general fiber of this fibration is an $(n+m)$-torus, and it degenerates where $\xi$ is in the image of the map
\begin{align}\label{eq:pants-log}
  \Log:\{1+u_1+\cdots +u_n = 0\mid u_i\in \CC^\times\}\to \RR^{n}, && \vec{u}\mapsto (\log|u_1|,\ldots, \log|u_n|)
\end{align}
and at least one of the $\eta_i$ is zero.

As a symplectic hypersurface, $X_{m,n}\subset \CC^{m+1}\times(\CC^\times)^n$ inherits a Stein Liouville structure from the restriction of the ambient structure on $\CC^{m+1}\times (\CC^\times)^n$. For compatibility with the SYZ fibration, we would like equip $\CC^{m+1}$ and $(\CC^\times)$ with the Stein structures defined by the respective potential functions $\varphi_1$ and $\varphi_2.$
However, for our purposes, it will be useful to deform the Liouville forms on both $\CC^{m+1}$ and $(\CC^\times)^n.$ We begin with $\CC^{m+1}.$

The obvious problem with using the function $\varphi_1$ to define the Stein structure on $\CC^{m+1}$ is that it is not quite a Stein potential: the form $dd^c\varphi_1$ is not positive but only semipositive, vanishing when at least two of the coordinates $z_i$ are 0. We will fix this by adding a perturbing term to $\varphi_1.$

For $0\leq i\neq j\leq m,$ let $r_{ij}:=|z_i|^2+|z_j|^2$, and let $\chi:\RR\to \RR$ be a smooth cut-off function, such that for some fixed $\alpha_1>\alpha_0>0,$ the function $\chi$ satisfies
\[
\chi(t)=\begin{cases}
    1 \text{ where } t<\alpha_0,\\
    0 \text{ where } t>\alpha_1.
\end{cases}
\]

\begin{lemma}\label{lem:kahler-def}
    Choose $0<\epsilon\ll \alpha_1,$ and let $\defpot_1:\CC^{m+1}\to\RR$ be the function
    \[
    \defpot_1(z):=\varphi_1(z)+\epsilon\sum_{i\neq j} \chi(r_{ij})r_{ij}.
    \]
    Then the function $\defpot_1(z)$ is a K\"ahler potential on $\CC^{m+1}.$
\end{lemma}
\begin{proof}
    First consider the regions where $\chi(r_{ij})$ is constant: i.e., for each $i\neq j,$ either $r_{ij}>\alpha_1$ or $r_{ij}<\alpha_0.$ In each case positivity of $dd^c\defpot_1$ is clear. In the former case, the second term vanishes, but the first term already defines a positive $(1,1)$-form. In the latter case, each term is semi-positive, and the second term is positive in those directions in which the first term isn't, so the total sum is positive.

    Thus it remains to consider the case where $\alpha_0<r_{ij}<\alpha_1.$ In this region, the perturbation term in $\defpot_1$ may fail to define a semi-positive $(1,1)$-form, since the function $\chi(r_{ij})$ is not plurisubharmonic. However, by compactness of the interval $[\alpha_0,\alpha_1],$ we can choose $\chi$ such that its $C^2$ norm on $[\alpha_0,\alpha_1]$ is bounded by a constant $C$ (depending on $\alpha_1-\alpha_0$). Since the same is true of the function $r_{ij}$ in the region $\alpha_0<r_{ij}<\alpha_1,$ we conclude that there exists an overall constant bounding 
    $\left|dd^c\left( \chi(r_{ij})r_{ij}
    \right)\right|$ in the region $\alpha_0<r_{ij}<\alpha_1.$ On the other hand, in the region where $r_{ij}>\alpha_0,$ $|dd^c\varphi_1|$ is bounded below. Putting these together, we conclude that for $\epsilon$ sufficiently small, the total function $\defpot_1$ will continue to define a positive $(1,1)$-form in this region.
\end{proof}



From now on, we write $\defpot:\CC^{m+1}\times (\CC^\times)^n\to \RR$ for the perturbed function
\begin{equation}\label{eq:mb-pert}
\defpot(z,u):=\defpot_1(z)+\varphi_2(u).
\end{equation}
Lemma \ref{lem:kahler-def} ensures that this perturbation of $\varphi$ is the potential for a Stein structure on $\CC^{m+1}\times(\CC^\times)^n.$

Both the original function $\varphi$ and the deformation $\defpot$ have a Morse-Bott manifold of minima given by the $(n+m)$-torus
$\TT:=\{\eta_i=0, \xi_j=\ell\},$ 
and the remaining critical points of $\defpot$ also project under $\pi$ to the 
subspace $\cH=\{\eta=0\}\subset \RR^{n+m}.$
To simplify the determination of these higher-index critical points, we will imitate
\cite{M,A06,Nwms,GS17} in applying a deformation to the Liouville structure on $(\CC^\times)^n.$

Within $\cH,$ the singular locus of the fibration $\pi$ coincides with the {\em amoeba} $\cA,$ defined as the image of the projection \eqref{eq:pants-log}.
The amoeba $\cA$ contains within it a {\em spine}, the tropical hypersurface $\Pi$ defined as the corner locus of the function $\max(0,\xi_1,\ldots,\xi_n).$
Observe that the complement $\cH\setminus \Pi$ consists of $n$ components, one of which is the ``all-negative'' orthant, and each of the others can be reached from this one by crossing one of the hyperplanes $\{\xi_i = 0\}.$ 
The components of this complement may be equivalently be described as the $\Log$ images of the loci where one monomial in the function $1+\sum u_i$ dominates all the others. We will write $C_i$ for the component where $u_i$ dominates, and $C_0$ for the bottom-left orthant, where $1$ dominates.

We will deform the space $X_{m,n}$ by a symplectic isotopy which brings the amoeba closer to the spine $\Pi,$ making it easier to understand the Liouville dynamics of $X_{m,n}.$ This trick was first used in \cite{M} and later elaborated in \cite[\S 4]{A06}, which we follow here. For $0\leq i \leq n,$ we pick a {\em tailoring function} $\psi_i$ (which is called $\phi_\alpha$ in \cite{A06}) satisfying
\begin{align*}
  \psi_i(\xi) = 
  \begin{cases}
    0& \text{when }d(\xi,C_i) \leq \frac{\epsilon}{2}\\
    1& \text{when }d(\xi,C_i) \geq \epsilon,
  \end{cases}
  && \text{and} &&
  \sum_{j=1}^n \left|\frac{\partial \psi_i}{\partial \xi_j}\right| < \frac{4}{\epsilon},
\end{align*}
for some $\epsilon\ll \ell,$
and we define the family of maps $\{f_s:(\CC^\times)^n\to \CC\}_{0\leq s\leq 1},$ by
\begin{equation}\label{eq:tailored-pants-equation}
  f_s(u_1,\ldots,u_n):= (1-s\psi_0(u)) + \sum_{i=1}^n (1-s\psi_i(u))u_i.
\end{equation}

  Consider the family $P^{n-1}_s:=\{f_s=0\}\subset (\CC^\times)^n.$ 
	For $s=0,$ the hypersurface $P^{n-1}_0$ is the usual pair of pants $\{1+\sum_{i=1}^n u_i=0\},$ and when $s=1,$ for $\log|u_i|\ll 0,$ the coordinate $u_i$ does not appear in the defining equation of $P^{n-1}_1,$ so that in that region $P^{n-1}_1$ is a product of $\CC^\times_{u_i}$ with $P^{n-2}_1.$
	\begin{proposition}[{\cite[Proposition 4.2]{A06}}]
		Near $\{f_s=0\},$ we have $|\bar{\partial}f_s| < |\partial f_s|,$ and therefore the family $P^{n-1}_s:=\{f_s=0\}\subset(\CC^\times)^n,$ for $0\leq s\leq 1,$ is a family of symplectic hypersurfaces.
	\end{proposition}
	As a result, $f_s$ is a symplectic fibration in a neighborhood of each $P^{n-1}_s,$ so that we can use the symplectic connection to lift $\frac{\partial f_s}{\partial s}$ to a vector field $Y_s$ which will flow $P^{n-1}_0$ to $P^{n-1}_1.$ The essential result which allows us to integrate the vector field $Y_s$ is the following, which is due to the fact that the isotopy $f_s$ was constructed to modify only the subdominant terms in $f_0$:
	\begin{lemma}[{\cite[Lemma 4.11]{A06}}]\label{lem:boundedness}
		The vector field $Y_s$ is bounded: there is a constant $C$ so that $|Y_s|<C$ everywhere (where the norm $|-|$ is taken with respect to the K\"ahler metric from $(\CC^\times)^n$).
	\end{lemma}
	Lemma \ref{lem:boundedness} allows the application of the Moser trick:
	\begin{corollary}[{\cite[Lemma 4.13]{A06}}]\label{cor:moser}
		There exists a Hamiltonian time-dependent vector field $Y_s'$ on $(\CC^\times)^n,$ supported in a neighborhood of $P^{n-1}_s,$ with germ at $P^{n-1}_s$ $C^\infty$-close to $Y_s$, 
		which is bounded and hence
		integrates to a symplectic flow that maps $P^{n-1}_0$ to $P^{n-1}_1.$
	\end{corollary}

	\begin{definition}
  Following \cite{Nwms}, we call the space $P^{n-1}_1$ resulting from applying this isotopy the {\em tailored pants} and denote it also by $\tP^{n-1}.$ 
  The regions $\log|u_i|\ll 0 $ on $\tP^{n-1}$ will be called the {\em legs} of the tailored pants, and the projection of $\tP^{n-1}$ under the $\Log$ map will be denoted by $\tcA$ and called the {\em tailored amoeba}.
	Similarly, we denote by $\widetilde{X}_{m,n}\subset \CC^{n+1}\times (\CC^\times)^m$ the space obtained by applying this isotopy to $(\CC^\times)^m.$ We refer to $\widetilde{X}_{m,n}$ as the {\em tailored local model}.
\end{definition}
\begin{remark}
  The singularities of the SYZ fibration $\pi$ on $X_{m,n}$ coincide precisely with the amoeba $\cA\subset\cH,$ and the effect of the tailoring isotopy is to make that amoeba a better approximation of the tropical hypersurface $\Pi$. From the perspective of tropical geometry, it would be desirable to have a more elaborate modification of the SYZ fibration which makes the discriminant locus agree precisely with $\Pi.$ Such a modification appears difficult, but some progress can be found in \cite{RZ2,EM}.
\end{remark}

Like $X_{m,n},$ the space $\widetilde{X}_{m,n}$ inherits a Liouville structure from its embedding into $\CC^{m+1}\times(\CC^\times)^n$; by pulling back along the deformation relating these spaces, we can understand $X_{m,n}$ and $\tilX_{m,n}$ as the same underlying symplectic manifold $X_{m,n}$, carrying a 1-parameter family of Liouville structures $\lambda_s,$ such that $\lambda_0$ is the original Liouville structure on $X_{m,n}$ and $\lambda_1$ is the Liouville structure on $\tilX_{m,n}.$ We will use the latter Liouville structure in our computation of the skeleton. 

In order to justify this, we must check that this family of Liouville structures is a Liouville homotopy in the sense of \cite{Cieliebak-Eliashberg}:

\begin{definition}[{\cite[Definition 11.5]{Cieliebak-Eliashberg}}]\label{def:Liouville-isotopy}
    A family of Liouville structures $(X,\omega =d\lambda_s),$ for $s\in[0,1]$ is a {\em Liouville homotopy} if there exists a smooth family of exhaustions $X=\bigcup_{k=1}^\infty X_s^k$ by compact domains $X_s^k\subset X$ with smooth boundaries along which the Liouville vector field is outward-pointing.
\end{definition}
\begin{lemma}[{\cite[Lemma 11.6]{Cieliebak-Eliashberg}}]
    A family of finite-type Liouville manifolds $X_s=(X,\omega=d\lambda_s)$ is a Liouville homotopy if the closure $\overline{\bigcup_{s\in[0,1]}\on{Skel}(X_s)}$ of the union of their skeleta is compact.
\end{lemma}

\begin{remark}
    The point of Definition \ref{def:Liouville-isotopy} is to
    ensure that such a homotopy cannot modify the Liouville structure in an interesting way at infinity. To see that this is a desirable definition, observe (\cite[Exercise 4]{Courte-Liouville}) that any Liouville structure on $\RR^{2n}$ can be deformed to the standard one (in a manner not necessarily satisfying Definition \ref{def:Liouville-isotopy}), whereas the conical Liouville structure on $\RR^{2n}$ should not be regarded as equivalent to the cotangent Liouville structure on $\rT^*{\RR^n} = \RR^{2n}.$
\end{remark}

\begin{proposition}\label{prop:liouville-homotopy}
    The family of Liouville structures interpolating between $X_{m,n}$ and $\tilX_{m,n}$ is a Liouville homotopy.
\end{proposition}
\begin{proof}
    For $0\leq s\leq 1,$ write $X_{m,n}(s)$ for the (symplectically isomorphic) family of submanifolds of $(\CC^\times)^n$ interpolating between $X_{m,n}$ and $\tilX_{m,n}$ by the flow of the vector field $Y_s.$ Observe that the Liouville form $\lambda_s$ is the pullback to $X_{m,n}(s)$ of the Liouville form $d^c\defpot$ on $\CC^{m+1}\times (\CC^\times)^n,$ and therefore the Liouville vector field on $X_{m,n}(s)$ is gradient-like (under the metric inherited from the K\"ahler metric on the ambient $\CC^{m+1}\times (\CC^\times)^n$) for the function $\defpot$.

    It will therefore be sufficient to show that for $N\gg 0,$ the function $\defpot|_{X_s}$ has no critical points on $\defpot^{-1}((N,\infty)).$ This is obvious except where the $z_i$ coordinates are 0, so we may reduce to the problem of establishing the corresponding fact for $|\Log-\ell|^2|_{P^{n-1}_s},$ where  $P^{n-1}_s\subset (\CC^\times)^n$ is the partially tailored $(n-1)$-dimensional pants defined by the equation \eqref{eq:tailored-pants-equation}.

   In other words, 
   in terms of the compactification $(\CC^\times)^n\subset \PP^n,$
   we must show that there exists a neighborhood $\overline{U}\subset \PP^{n-1}$ of $\partial \PP^{n-1}$ 
   such that the gradient of $|\Log-\ell|^2$ on $P^{n-1}_s$ does not vanish within $U:=P^{n-1}_s\cap \overline{U}$ for any $s.$ Moreover, as this calculation, which is taking place at infinity, is insensitive to $\ell,$ we may as well take $\ell=0$ for simplicity.
   
   Let $a\in \overline{P}^{n-1}_s$ be a point in the compactification of $P^{n-1}_s$ inside $\PP^n.$ By symmetry, we may assume that $a\in \CC^n$. This point is contained in $k$ boundary divisors, for some $1\leq k \leq n-1$; again by symmetry, we may write $a=(a_1,\ldots,a_n)$ where $a_1=\cdots=a_k = 0,$ and $a_{k+1},\ldots,a_n\neq 0,$ 
   and we may take $a_n$ to be the largest of the $a_i$. 
   Since at least one $a_i$ must satisfy $|a_i|>\frac{1}{n},$ we know in particular that $|a_n|>\frac{1}{n}$. Note also that in  a neighborhood of $a$, the final monomial in $1+u_1+\cdots+u_n$ is a dominant term, so we have $\psi_n$ identically 0 near $a,$ where $\psi_n$ is the cut-off function used to define $f_s.$

   We will now show that for small $\epsilon$, the gradient of $|\Log|^2$ on $P^{n-1}_s$ does not vanish for any $u=(u_1,\ldots,u_n)\in P^{n-1}_s$ within an $\epsilon$-neighborhood of $a$ -- i.e., such that $|u_i-a_i|<\epsilon$ for all $i$.
   Observe that our assumptions ensure that for $u$ near $a$, the germ of curve $(u_1+tu_1,u_2,\ldots,u_{n-1},u_n+tu_1(1-s)\psi_1)$ at $u$ lies in $P^{n-1}_s$ and thus defines a tangent vector $V\in T_u(P^{n-1}_s)$; in real coordinates $u_i=r_ie^{\sqrt{-1}\theta_i},$ we may write this tangent vector as
   \[
   V=\partial_{r_1}+(C\partial_{r_n}+C'\partial_{\theta_n})
   \]
   for some real constants $C,C'.$
   Since $|u_n|>\frac{1}{n}-\epsilon$ and $|u_1|<\epsilon,$ and $|\psi_1|<1,$ we see that
   by choosing $\epsilon$ sufficiently small, we may ensure $|C|,|C'|\ll 1$ for all $s.$
   We would like a lower bound for $|d|\Log|^2|$ evaluated on $V$.

   Since $d|\Log|^2 = \sum d(\log|u_i|)^2 = \sum \frac{\log(r_i)}{r_i}dr_i,$ we have 
   \begin{align}
   \left|\left(d|\Log|^2\right)(V)\right|&= \nonumber
   \left|\frac{\log(r_1)}{r_i}+C\frac{\log(r_n)}{r_n}\right|\\ 
   &\geq \left|\frac{\log(r_1)}{r_i}\right|-C\left|\frac{\log(r_n)}{r_n}\right|.
   \label{eq:difference-gradbound}
   \end{align}
   Since $0<r_1<\epsilon \ll 1,$ the first term is at least $|\log(\epsilon)/\epsilon|,$ which is large and grows monotonically as $\epsilon$ decreases, whereas
   since $r_n>\frac{1}{n}-\epsilon$, the term $|\log(r_n)/r_n|$ is bounded, and since $C$ can be made arbitrarily small by decreasing $\epsilon,$ we conclude that
   the difference \eqref{eq:difference-gradbound}
 is bounded away from 0, as desired.
 \end{proof}

\begin{remark}
    The proof of Proposition \ref{prop:liouville-homotopy} may be summarized as, ``apply the usual proof that the hypersurface $X_{m,n}$ is a finite-type Stein manifold, and check that it continues to hold uniformly over the deformation.'' In fact, this first statement (viz., that the Stein structure on an affine algebraic variety is finite-type) is not well-documented in the literature; it appears to have been a folklore theorem for some time, although a proof can be found as Lemma 8 of \cite{seidel-smith-ramanujam}. In any case, the second step in this case (showing that the deformation remains finite-type) requires almost no extra work, since increasing $s$ only improves the bound on \eqref{eq:difference-gradbound} by taking $C$ closer to 0.
\end{remark}


With the tailored local model $\widetilde{X}_{m,n}$ in hand, we are ready to study the Liouville geometry of $\tX_{m,n}$. 
As the Liouville structure is gradient-like for the restriction
(along the inclusion $\tX_{m,n}\subset \CC^{m+1}\times(\CC^\times)^n$) of the function $\defpot$ defined in \eqref{eq:mb-pert},
we are reduced to studying the Morse-Bott theory of this function.
Our Morse-Bott calculation will imitate that performed in \cite{Nwms,GS17}, relying heavily on the inductive structure of the tailored pants $\tP^{n-1}.$

\begin{proof}[Proof of Theorem \ref{thm:main}]
  It will be useful to study the potential $\defpot$ via the singular torus fibration \eqref{eq:SYZ}. As we have noted, we may restrict our attention to the locus $\cH = \{\eta = 0\}= \RR^n\subset \RR^{m+n}.$ (For the remainder of this proof, we will always take the coordinates $\eta_i$ to be equal to 0.) We will begin from the index-0 critical locus, which we denote by
\[
	\TT_\emptyset:=\{\xi=(-\ell,\ldots,-\ell)\},
\]
and attach stable manifolds for the critical points of progressively higher index.

  Fix some $i\in\{1,\ldots,n\}$ and consider the region in $\cH$ where $\xi_i=0$ and $\xi_j\ll 0$ for $j\neq i.$ In this region, the tailored amoeba $\tcA^n$ is a hyperplane, and the tailored pants $\tP^{n-1}$ is defined by the equation $\{u_i+1=0\}.$ It is clear that in this region, the function $\defpot$ has an index-1 Morse-Bott critical locus at the $(n-1)$-torus 
  \[\TT_i:=\{u_i=-1,\log|u_j| = -\ell\mid j\neq i\}.\]
  The stable set for this critical locus projects to $\cH$ along the interval $\{0\geq \xi_i\geq -\ell,\xi_j=-\ell\mid j\neq i\}$ (or possibly a small perturbation of this interval, depending on how the bump function $\chi$ used in defining $\defpot$ is chosen), and it attaches to $\TT_\emptyset$ as the Legendrian whose front projection is the codimension-1 subtorus obtained by fixing $\Arg(u_i).$

  The index-$k$ critical manifolds, for $1<k\leq n,$ are computed similarly. We fix some $I=\{i_1,\ldots,i_k\}\subset [n]$ and consider the region where $\xi_i=0$ for $i\in I$ and $\xi_j \ll 0$ for $j\notin I.$ In this region, the tailored pants $\tP^n$ looks like a product $(\CC^\times)^{n-k}\times \tP^{k-1},$ and the calculation of \cite[Theorem 5.13]{Nwms}, repeated in more detail in the proof of \cite[Theorem 5.3.4]{GS17}, shows that there is a single critical manifold in this region, lying over the point in the boundary of the tailored amoeba $\tcA$ where $\xi_i=\xi_j$ for all $i,j\in I,$ and $\xi_j = -\ell$ for all $j\notin I.$ This critical manifold $\TT_I$ is a $(n-k)$-torus, obtained by fixing the values of $u_i$ for $i\in I$ and $\log|u_j|$ for $j\notin I.$

  Moreover, the attachments for these critical points are just as described in \cite[Theorem 5.3.4]{GS17}, giving the FLTZ skeleton of the pants, attaching to $\TT_\emptyset$ in the expected way.
\end{proof}

\begin{example}Let $n=2.$ Then the region $\cH$ in the SYZ base is a plane, depicted in Figure \ref{fig:syz-skel} together with the tailored amoeba $\tcA$ (containing the tropical curve $\Pi$ for reference).
  \begin{figure}[h]
    \begin{center}
      \includegraphics[width=2in]{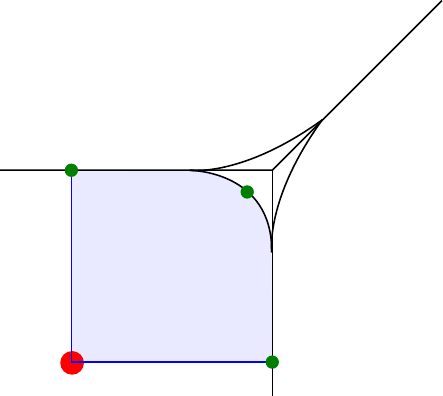}
    \end{center}
    \caption{The plane $\cH,$ which lives inside the SYZ base of $X_{m,2}.$ The Lagrangian skeleton is the union of the torus fiber over the red point and a piece projecting to the blue shaded region. The higher-index critical points are depicted in green.}
    \label{fig:syz-skel}
  \end{figure}
  Near the torus fiber over the red point, the skeleton looks like the product of the FLTZ skeleton $\LL_{\CC^2}$ with an $m$-torus $T^m$ parametrized by $\Arg(z_1),\ldots,\Arg(z_m).$ This $T^m$ factor collapses to a point over points in the base where the blue region meets the boundary of the amoeba.
\end{example}

\section{SYZ geometry and cluster theory}\label{sec:trop}
In this section, we connect the results of this paper to the non-toric blowup construction of Gross-Hacking-Keel, and explain how our results fit into a program to understand mirror symmetry for a general affine cluster variety.

In \cite{GHK1,GHK2}, it was proposed that a cluster variety $X$ should be studied via its relation to a {\em toric model}: a toric variety $\oX$ containing a codimension-2 subvariety $H$ inside its toric boundary divisor $D$ such that $X$ embeds in the blowup $\Bl_H\oX$ as the complement of the proper transform $\widetilde{D}$ of the toric boundary divisor $D.$ 
The requirement that the result $X$ of this construction is a cluster variety imposes strong constraints on $H,$ which must be defined by a character of the dense torus in its component of $D$. 

In the case at hand, we relax this constraint, allowing $H$ to be defined by a general linear polynomial. We also allow contributions from the places where $H$ meets the singularities of the boundary divisor $D$. With these modifications to the above construction, we can produce our spaces $X_{n,m}.$

Let $\oX=\CC^n\times (\CC^\times)^m,$ with coordinates $z_1,\ldots,z_n,u_1,\ldots,u_m.$ 
Write $D:=\{z_1\ldots z_n=0\}\subset \oX$ for the toric boundary divisor, and consider the non-toric subvariety $H=\{z_1\ldots z_n=0, \sum u_i=-1\}\subset D$.  
\begin{proposition}\label{prop:gen-GHK}
  Let $X=\Bl_H\oX\setminus \widetilde{D}$ be the result of blowing up $H$ and then deleting the proper transform of the boundary divisor $D$. Then $X\cong X_{n,m}$.
\end{proposition}
\begin{proof}
  Let $f = z_1\cdots z_n$ and $g=1+u_1+\cdots + u_m.$ The center $H$ is defined by the equations $f=g=0,$ so that blowing up $H$ involves introducing $\frac{f}{g}, \frac{g}{f}$ into the ring of functions of $\oX.$ Deleting the proper transform of $D$ means it is sufficient to introduce the function $\frac{g}{f}$ into the coordinate ring. If we denote this new variable by $z_0,$ then the coordinate ring of $X$ is $\CC[z_1,\ldots,z_n,u_1^\pm,\ldots,u_m^\pm][z_0]/(z_0 f = g) = \CC[X_{n,m}].$
\end{proof}

The equation $z_0\cdot z_1 \cdots z_n = 1 + u_1 + \cdots u_m$ should be understood as a generalization of the usual ``cluster transformation equation'' $z_0z_1 = 1 + u.$ Note that in this simpler case, the variable $z_1$ is obtained from $z_0$ by a cluster mutation, whereas in the more general situation considered in this paper, there are $n+1$ variables $z_i$, corresponding to $n$ possible mutations.

On the symplectic side, this corresponds to the fact that the singular locus $\cA\subset \cH$ divides $\cH$ into $n$ chambers. In the constructions of \S \ref{sec:skel-calc}, we distinguished one of those chambers (the ``all-negative'' orthant), but we could equally have begun with a torus fiber in one of the other chambers instead. The difference between the torus fibers in the various chambers is a generalization of the familiar ``Clifford/Chekanov'' distinction in cluster symplectic geometry.

On the B-side, as we have seen, this new feature is due to the fact that the center $H$ of the GHK blowup intersects the singular locus of the divisor $D.$ Keeping track of contributions from these intersections is a crucial step in generalizing the results of \cite{GL}, which disregarded those contributions.

The results of this paper therefore suggest a strategy for studying mirror symmetry for an affine cluster variety $X$ via its SYZ fibration. The SYZ base should have a presentation by gluing together local pieces modeled by the singularities $\mathfrak{s}_{n,m}$ (or products of these). Each singularity has a monodromy-invariant space $\cH,$ and if there is a point in the intersection of all of these spaces, it should play the r\^ole of the torus $\TT_\emptyset$ in the skeleton $\LL$ of $X,$ with the remainder of the skeleton obtained by attaching and collapsing subtori in a manner prescribed by the geometry of the base, locally agreeing wih the construction in this paper (as illustrated in Figure \ref{fig:syz-skel}). This presentation should match a presentation of the mirror space $X^\vee$ in terms of the non-toric blowup construction of Proposition \ref{prop:gen-GHK}.
\vspace{2mm}

{\bf Acknowledgements.} This work is the direct outgrowth of a continuing project with Ian Le on cluster varieties, and my understanding of the generalized non-toric blowup construction of Proposition \ref{prop:gen-GHK} is due to conversations with him. My interest in the spaces $X_{n,m}$ comes from conversations with Vivek Shende about the r\^ole of these spaces in compactifying the large-volume manifolds described in \cite{GS20} (to be explained in forthcoming work), and I also thank Denis Auroux for advice and encouragement on the program sketched in the preceding paragraphs. I also was helped by conversations with Maxim Jeffs about Liouville isotopies, and I thank an anonymous referee for careful reading and suggestions which very much improved the paper. I am supported by an NSF Postdoctoral Research Fellowship, DMS-2001897.

\bibliographystyle{plain}
\bibliography{refs}
\end{document}